\documentclass[12pt]{article}

\usepackage{amsmath}
\usepackage{amssymb}
\usepackage[mathscr]{eucal}
\usepackage{graphicx}
\usepackage{color}

\usepackage{bm}
\usepackage{mathrsfs}
\usepackage{amsthm}
\usepackage{enumerate}
\usepackage[mathscr]{eucal}
\usepackage{eqlist}
\usepackage{graphicx}

\newtheorem{Theorem}{\sc Theorem}
\newtheorem{Definition}[Theorem]{\sc Definition}
\newtheorem{Proposition}[Theorem]{\sc Proposition}
\newtheorem{Lemma}[Theorem]{\sc Lemma}

\newtheorem{Remark}[Theorem]{\sc Remark}
\newtheorem{Example}[Theorem]{\sc Example}
\newtheorem{Problem}[Theorem]{\sc Problem}

\newcommand{\R}{{\if mm {\rm I}\mkern -3mu{\rm R}\else \leavevmode
\hbox{I}\kern -.17em\hbox{R} \fi}}

\def\sqr#1#2{{
    \vcenter{
         \vbox{\hrule height.#2pt
               \hbox{\vrule width.#2pt height#1pt \kern#1pt
                     \vrule width.#2pt
               }
               \hrule height.#2pt
         }
    }
}}

\def\bar{\overline}
\def\real{\mathbb{R}}
\def\nat{\mathbb{N}}

\def\lista#1
{{ \itemindent 0.0cm \labelsep .2cm \leftmargin 0.8cm \rightmargin
0.0cm \labelwidth 0.6cm \topsep 0.0mm
\parsep 0.0mm
\itemsep 0.0mm
\begin{list}{}
{ \setlength{\leftmargin}{.8cm} \setlength{\rightmargin}{0.0cm}
\setlength{\parsep}{0.0mm} \setlength{\topsep}{.0mm}
\setlength{\parskip}{.0cm} \setlength{\itemsep}{.0cm} }
{#1}\end{list}} }

\begin{document}
\title{
A Class of Elliptic Quasi--Variational--Hemivariational Inequalities with Applications
}
		
\author{
Stanis{\l}aw Mig\'orski\footnote{\,College of Applied Mathematics, Chengdu University of Information Technology, Chengdu, 610225, Sichuan Province, P.R. China, and  Jagiellonian University in Krakow, Chair of Optimization and Control, ul. Lojasiewicza 6, 30348 Krakow, Poland.
Tel.: +48-12-6646666. E-mail address: stanislaw.migorski@uj.edu.pl.},\ \ Jen-Chih Yao\footnote{\, Center for General Education, China Medical University, Taichung, Taiwan. E-mail address: yaojc@mail.cmu.edu.tw.}
\ \ \mbox{and} \ \
Shengda Zeng\footnote{\,Guangxi Colleges and Universities Key Laboratory of Complex System Optimization and Big Data Processing, Yulin Normal University, Yulin 537000, Guangxi, P.R. China, and Jagiellonian University in Krakow,
Faculty of Mathematics and Computer Science, ul. Lojasiewicza 6, 30348 Krakow, Poland.
Tel.: +86-18059034172. Corresponding author.
E-mail address: zengshengda@163.com.}
}

\date{}
\maketitle
	
\noindent {\bf Abstract.} \
In this paper we study a class of quasi--variational--hemi\-va\-ria\-tio\-nal inequalities
in reflexive Banach spaces.
The inequalities contain a convex potential,
a locally Lipschitz superpotential, and a solution-dependent
set of constraints.
Solution existence and compactness of the solution set to the inequality problem are established based on the Kakutani--Ky
Fan--Glicksberg  fixed point theorem. Two examples of the interior and boundary semipermeability models illustrate the applicability of our results.

\smallskip
	
	
\smallskip
	

\section{Introduction}

In this paper we study a quasi--variational--hemivariational inequality of elliptic type of the following form.
Find $u \in C$ such that $u \in K(u)$ and
\begin{equation}\label{AAA}
	\langle A u - f, z - u \rangle
	+ \varphi (z) - \varphi (u)
	+ j^0(Mu; Mz - Mu) \ge 0
	\ \ \mbox{\rm for all} \ \ z \in K(u).
	\end{equation}
\noindent
Here,
$C$ is a subset of a reflexive Banach space $V$,
$A \colon V \to V^*$ is a nonlinear operator,
$K \colon C \to 2^C$ is a set-valued mapping,
$M \colon V \to X$ is linear and compact,
where $X$ is a Hilbert space,
$\varphi \colon V \to \real$ is a convex function (potential),
$j \colon X \to \real$ is a locally Lipschitz
potential where $j^0(u; z)$ denotes the generalized directional derivative of $j$ at point $u\in X$ in the direction $z\in X$,
and $f \in V^*$.

The motivation to study inequality (\ref{AAA}) comes from modeling
of physical phenomena for semipermeability problems of flow through porous media:
monotone relations are described by variational
inequalities~\cite{Barbu,DL},
and nonmonotone laws are governed by hemivariational
inequalities~\cite{MO2004,NP,Pana1,P}, see bibliographies therein.
In these references, semipermeability on the boundary or in the domain is considered.
Recently, various particular forms of inequality (\ref{AAA})
have been treated for nonconvex semipermeability problems in~\cite{MKZ2020,migorskibaizengNA2019,zengmigorskiliuyao2020} as well as for nonsmooth contact problems of mechanics of solids and fluids, see~\cite{DuMi,Hung2022JCAM,Hung2020FSS,SofMatei,SMHistory,zengmigorskikhanSICON2021}.

To highlight the general form of our problem, we list the following particular cases of problem (\ref{AAA}).
\begin{itemize}
\item[(i)] If $j = 0$ and $K$ is independent of the solution, then problem~(\ref{AAA}) reduces to the elliptic variational inequality of the first kind studied in~\cite{LS,SofMatei}:
$$
u \in K, \quad
\langle A u, v - u \rangle + \varphi (v) - \varphi (u) \ge \langle f, v - u \rangle \quad \mbox{for all} \ v \in K.
$$

\item[(ii)]
When $j=0$ and $K=V$, problem~(\ref{AAA}) takes the form of the elliptic variational inequality of the second kind investigated in~\cite{Brezis,LS,Mosco,SofMatei}:
$$
u \in V, \quad
\langle A u, v - u \rangle + \varphi (v) - \varphi (u) \ge \langle f, v - u \rangle \quad \mbox{for all} \ v \in V.
$$

\item[(iii)]
For $j=0$, problem~(\ref{AAA}) reduces to the elliptic quasi-variational inequality treated in~\cite{MKZ}:
find $u\in C$ such that $u\in K(u)$ and
$$
\langle Au,v-u\rangle +\varphi(v)-\varphi(u)
\ge \langle f,v-u\rangle
\ \ \mbox{for all} \ \ v\in K(u).
$$

\item[(iv)]
While $j = 0$, $K$ independent of $u$,
and $\varphi \equiv 0$, then
problem~(\ref{AAA}) becomes the elliptic variational inequality of the form
$$
u \in K, \quad
\langle A u, v - u \rangle \ge
\langle f, v - u \rangle
\ \ \mbox{for all} \ \ v \in K,
$$
which has been considered in \cite{Brezis,Browder,Kind-St,LS, Mosco}.
\item[(v)]
If $\varphi=0$ and $K =V$, problem~(\ref{AAA})
takes the form of elliptic hemivariational inequality of the form
$$
u \in V, \quad
\langle A u, v \rangle + j^0(u; v) \ge \langle f, v \rangle \quad \mbox{for all} \ v \in V,
$$
\noindent
this has been
studied in~\cite{NP}.
\end{itemize}

Elliptic variational and quasi-variational inequalities
have been studied in~\cite{khanGiannessi2000,Khan,LZ},
variational-hemivariational inequalities have been considered  in~\cite{liu4,liu2,MOS30,ZLM}, while
quasi--hemivariational inequalities have been treated
only recently in~\cite{MKZ2020,zengmigorskikhanSICON2021},
and applications to implicit obstacle problems can be found in~\cite{MKZ,MNZ}.

We underline that under the relaxed monotonicity condition of the subgradient operator $\partial j$ (see Remark~\ref{RRCC} in Section~\ref{section4}), problem~(\ref{AAA}) has been studied only recently in~\cite{MKZ2020} by the Kluge fixed point theorem applied to a set--valued variational selection for the set of constraints.
However, a question arises on how to prove the existence of a solution to the elliptic quasi--variational--hemivariational inequality (\ref{AAA}) without the relaxed monotonicity condition
on the subgradient operator $\partial j$.
The main aim of the current paper is to  give a positive
answer to this open question.
Besides, the method applied in this paper is different
from that used in~\cite{MKZ2020}. More precisely, our approach
is based on the Kakutani--Ky
Fan--Glicksberg  fixed point theorem and a convergence theorem for variational inequalities.
Note that two classes of general semipermeability problems studied in Section~\ref{section5} lead in their weak formulation to inequality (\ref{AAA}) where the operator $M$ is either the embedding operator or the trace operator. The second case for boundary semipermeability problems can not be treated by the approach used previously in~\cite{MKZ2020}.

Finally, we also stress that our results are applicable
to a wide spectrum of problems met in contact mechanics, e.g.,
a nonlinear elastic contact problem with normal compliance condition with a unilateral constraint, and a contact problem with the Coulomb friction law in which the friction bound is supposed to depend on
the normal displacement,
studied, e.g., in~\cite{zengmigorskikhanSICON2021}.
For various interesting applications, we refer
to~\cite{AlphonseHintermullerRautenbergCVPDEs2018,BC,DL,GrunewaldKimCVPDEs2019, HMS,HMS_ED,DuMi, SofMatei}.

\section{Notation and preliminary material}\label{Preliminaries}

\noindent
In this section we fix a basic notation and recall some concepts and  facts we need in this paper, details can be found  in~\cite{Clarke,DMP1,MOSbook}.

Let $(Y, \|\cdot\|_Y)$ be a Banach space,
$Y^*$ be its dual space and  $\langle\cdot,\cdot\rangle$ denote the duality bracket between $Y^*$ and $Y$.
Given a locally Lipschitz function
$j \colon Y \to \real$,
the generalized subgradient of $j$ at $x\in Y$
is defined by
\begin{equation*}
\partial j (x) = \{\, x^* \in Y^* \mid
{\langle x^*, v \rangle} \le j^{0}(x; v)
\ \ \mbox{for all} \ \ v \in Y \, \},
\end{equation*}
where the generalized directional derivative
of $j$ at $x \in Y$ in the direction $v \in Y$ is
given by
\begin{equation*}
j^{0}(x; v) = \limsup_{y \to x, \ \lambda \downarrow 0}
\frac{j(y + \lambda v) - j(y)}{\lambda}.
\end{equation*}
Let $\phi \colon Y \to \real$ be a convex function and
$x \in Y$. An element $x^* \in Y^*$ is called a subgradient of $\phi$ at $x$ if and only if the following inequality holds
$$
\phi (v) \ge \phi (x) + \langle x^*, v - x\rangle
\ \ \mbox{for all} \ \ v \in Y.
$$
The set of all $x^* \in Y^*$ that satisfy this inequality is called
the subdifferential of $\phi$ at $x$, and it denoted by
$\partial_c \phi (x)$.

We recall the basic definitions for single-valued operators. Let $A \colon Y \to Y^*$.
%
%
An operator $A$ is called bounded if it maps bounded
sets of $Y$ into bounded sets of $Y^*$.
%
%
$A$ is called monotone if
\begin{eqnarray*}
\langle Av_1 - A v_2, v_1-v_2 \rangle \ge 0
\mbox{ for all }v_1, v_2 \in Y.
\end{eqnarray*}
%
%
An operator $A$ is called $m$-strongly monotone if
\begin{eqnarray*}
\langle Av_1 - A v_2, v_1-v_2 \rangle \ge
m \| v_1 - v_2\|^2
\mbox{ for all }v_1, v_2 \in Y.
\end{eqnarray*}
%
%
$A$ is called coercive if
\begin{eqnarray*}
\langle A v, v\rangle \ge \alpha(\| v \|)\| v \|
\mbox{ for all } v \in V,
\end{eqnarray*}
where $\alpha \colon \real_+ \to \real$ is such that $\alpha(s) \to +\infty$ as $s \to +\infty$.
%
%
An operator $A$ is called pseudomonotone,
if it is bounded and
if $u_n \to u$ weakly in $Y$ and
$\displaystyle \limsup \langle A u_n, u_n -u \rangle \le 0$
implies
\begin{eqnarray*}
\displaystyle \langle A u, u - v \rangle
\le \liminf \langle A u_n, u_n - v \rangle\mbox{ for all }v \in Y.
\end{eqnarray*}

The space of linear bounded operators from
a Banach space $E$ to a Banach space $F$
is denoted by ${\mathcal L}(E, F)$,
it is a Banach space endowed
with the usual norm
$\| \cdot \|_{{\mathcal L}(E, F)}$.
For a subset $D$ of a normed space $Y$,
we write
$\| D \|_Y = \sup \{ \| u \|_Y \mid u \in D \}$.
The symbol $Y_w$ is used for the space $Y$ endowed with the weak topology, and $2^Y$ stands for the set of all subsets of $Y$.
For a set--valued map $F\colon X_1 \to 2^{X_2}$ between topological spaces $X_1$ and $X_2$,
the graph of $F$ is the set
${\rm Gr}(F) = \{ (x, y) \in X_1 \times X_2 \mid y \in F(x) \}$.

%

We recall a definition of Mosco convergence, see~\cite{Mosco}.
\begin{Definition}\label{DefMosco}
Given a normed space $Y$, a sequence $\{ C_n \}$ of closed and convex sets in $Y$, is said to converge to a closed and convex set $C \subset Y$ in the Mosco sense, denoted by
$C_n \ \stackrel{M} {\longrightarrow} \ C$
as $n \to \infty$, if we have

\medskip

{\rm (m1)}\ \
if $\{ n_k \}$ is a sequence of indices converging to $\infty$,
$\{ z_k \}$ is a sequence

\qquad \ \
such that $z_{k} \in C_{n_k}$ for every $k$ and $z_{k} \to z$ weakly in $Y$, then
$z \in C$,

\smallskip

{\rm (m2)}\ \
for every $z \in C$, there exists a sequence $z_n \in C_n$ with $z_n \to z$ in $Y$.
\end{Definition}

There is an alternative definition of Mosco convergence defined in terms of Kuratowski limits~\cite[Section 4.7]{DMP1}.
We end the section by recalling the  Kakutani--Ky
Fan--Glicksberg  theorem for a reflexive Banach space,
see e.g.~\cite[Theorem~2.6.7]{PK}.
\begin{Theorem}\label{KKF}
Let $Y$ be a reflexive Banach space and
$D \subseteq Y$ be a nonempty, bounded, closed and convex set. Let $\Lambda \colon D \to 2^D$
be a set-valued map with nonempty, closed and convex values such that its graph is sequentially closed in $Y_w \times Y_w$ topology.
Then $\Lambda$ has a fixed point.
\end{Theorem}

\section{Well-posedness result for variational  inequalities}\label{section3}

In this section, we formulate an abstract elliptic variational inequality with constraints
for which we provide results on its unique solvability
and convergence under the perturbation on the data.
We will need such results in Section~\ref{section4} to investigate a class of quasi--variational--hemivariational inequalities.

Let $Y$ be a reflexive Banach space with a norm $\| \cdot \|$, $Y^*$ be its dual space, $\langle \cdot, \cdot \rangle$ denote the
the duality brackets for the pair $Y^*$ and $Y$, $E$ be a nonempty, closed and convex subset of $Y$.

Consider the following elliptic variational inequality of the first kind.
\begin{Problem}\label{Problem1}
{\it Find an element $u \in E$ such that}
\begin{equation*}
\langle A u - g, z - u \rangle
+\varphi (z) - \varphi (u)
\ge 0 \ \ \mbox{\rm for all} \ \ z \in E.
\leqno{(P)}
\end{equation*}
\end{Problem}
We impose the following hypotheses on the data of Problem~\ref{Problem1}.
\begin{eqnarray}\label{Hyp1}
\left\{\begin{array}{lll}
A \colon Y \to Y^* \ \mbox{is pseudomonotone and $m$-strongly monotone}, \\ [2mm]
E \ \mbox{is a nonempty, closed, convex subset of} \ Y, \\[2mm]
\varphi \colon Y \to \real \ \mbox{is convex and lower semicontinuous}, \\ [2mm]
g \in Y^*.
\end{array}\right.
\end{eqnarray}

\begin{Lemma}\label{Lemma1}
Under the hypothesis {\rm (\ref{Hyp1})}, Problem~\ref{Problem1} has a unique solution $u \in E$.
\end{Lemma}

\begin{proof}
It is a consequence
of~\cite[Theorem~3.2]{liuzengmigorski}.
We only note that if $A$ is strongly monotone with constant $m > 0$, i.e.,
$$
\langle A v_1 - Av_2, v_1 - v_2 \rangle
\ge m\, \| v_1 - v_2 \|^2
\ \ \mbox{for all} \ v_1, v_2 \in Y,
$$
then $A$ is coercive too:
$$
\langle A v, v \rangle =
\langle Av - A0, v \rangle + \langle A0, v \rangle \ge
m \, \| v \|^2 + \| A0 \|_{Y^*}\| v \|
$$
for all $v \in Y$.
\end{proof}

The following result is the well--known Minty formulation and it follows from~\cite[Lemma~4.1]{Mosco}.
\begin{Lemma}\label{Lemma2}
Let $A\colon Y \to Y^*$, $E \subset Y$,
$\varphi \colon Y \to \real$ and $g \in Y^*$.
\begin{itemize}
\item[{\rm(i)}]
If $A$ is monotone, then any solution to the inequality
\begin{equation}\label{alpha1s}
u \in E: \
\langle A u - g, z - u \rangle
+\varphi (z) - \varphi (u)
\ge 0 \ \ \mbox{\rm for all} \ \ z \in E
\end{equation}
is also a solution to the following problem
\begin{equation}\label{beta1s}
u \in E: \
\langle A z - g, z - u \rangle
+\varphi (z) - \varphi (u)
\ge 0 \ \ \mbox{\rm for all} \ \ z \in E.
\end{equation}
\item[{\rm(ii)}]
If $A$ is hemicontinuous, $E$ is convex and $\varphi$ is convex, then any solution of (\ref{beta1s}) is a solution to (\ref{alpha1s}).
\end{itemize}
\end{Lemma}

We will show the result of the continuous dependence of solution to Problem~$(P)$ on the data $(E, g)$. To this end,
we consider the perturbed elliptic variational inequality
$(P_n)$ for $n \in \nat$.
\begin{Problem}\label{Problem1n}
{\it Find an element $u_n \in E_n$ such that}
	\begin{equation*}
	\begin{array}{lll}
	\langle A u_n - g_n, z - u_n \rangle
	+\varphi (z) - \varphi (u_n)
	\ge 0 \ \ \mbox{\rm for all} \ \ z \in E_n.
	\end{array}
	\leqno{(P_n)}
	\end{equation*}
\end{Problem}
Also, we impose the following hypotheses.
\begin{eqnarray}\label{Hyp1n}
\left\{\begin{array}{lll}
(a) \
E_n \ \mbox{are nonempty, closed, and convex subsets of} \ Y \mbox{ such that }
E_n \ \stackrel{M} {\longrightarrow} \ E,
\\[2mm]
(b) \
g_n \in Y^*, \ g_n \to g \ \mbox{in} \ Y^*.
\end{array}\right.
\end{eqnarray}

\begin{Proposition}\label{Prop1}
Under hypotheses (\ref{Hyp1}) and (\ref{Hyp1n}),
for each $n \in \nat$, Problem~$(P_n)$ 
has a unique solution $u_n \in E_n$
such that the sequence $\{ u_n \}$ converges in $Y$,
as $n \to \infty$, to the unique solution $u \in E$ of Problem~(P).
\end{Proposition}

\begin{proof}
It suffices to apply Lemma~\ref{Lemma1} with $E = E_n$
and $g = g_n$ for every $n \in \nat$,
to deduce that Problems~$(P_n)$
and~$(P)$
have unique solutions $u_n \in E_n$ and $u \in E$, respectively.

We show that the sequence $\{ u_n \} \subset E_n$
is bounded in $Y$.
From the condition (m2) of the Mosco convergence in Definition~\ref{DefMosco},
there exists $w_n \in E_n$ such that
$w_n \to u$ in $Y$, as $n \to \infty$.
We choose $z = w_n \in E_n$ in Problem~$(P_n)$
to get
\begin{equation*}
\langle A u_n, u_n - w_n \rangle \le
\langle g_n, u_n - w_n \rangle
+\varphi (w_n) - \varphi (u_n).
\end{equation*}
Combining the latter with the $m$-strong monotonicity of $A$ implies
\begin{eqnarray}
&&
m\, \| u_n - w_n \|^2 \le
\langle A u_n - A w_n, u_n - w_n \rangle
=
\langle A u_n, u_n - w_n \rangle
- \langle A w_n, u_n - w_n \rangle \nonumber
\\[2mm]
&&\quad
\le
\langle g_n, u_n - w_n \rangle
+ \varphi (w_n) - \varphi (u_n)
- \langle A w_n, u_n - w_n \rangle \nonumber
\\ [2mm]
&&\qquad
=
\langle A w_n - g_n, w_n - u_n \rangle
+ \varphi (w_n) - \varphi (u_n).
\label{star1}
\end{eqnarray}
Since, by (\ref{Hyp1}),
the function $\varphi$ is convex, lower
semicontinuous and finite, it is continuous on $Y$, see~\cite[Theorem~5.2.8]{DMP1}, and hence
$\varphi(w_n) \to \varphi (u)$, as $n \to \infty$,
and
\begin{equation}\label{1***}
|\varphi (w_n) | \le
|\varphi(w_n) - \varphi (u)| + |\varphi(u)|
\le 1 + |\varphi(u)|
\end{equation}
for sufficiently large $n$. On the other hand, the function $\varphi$ has an affine minorant,
see~\cite[Proposition~5.2.25]{DMP1}, that is,
there exist $\l \in Y^*$ and $b \in \real$ such that
$\varphi(z) \ge \langle l, z \rangle + b$
for all $z \in Y$. Hence, we have
\begin{equation}\label{*3}
- \varphi(u_n) \le \| l\|_{Y^*} \| u_n \| + |b|. \end{equation}
Now, using (\ref{star1}), (\ref{1***}), and (\ref{*3}), we obtain
\begin{equation*}
m\, \| u_n - w_n \|^2 \le
\| Aw_n - g_n \|_{Y^*} \| u_n - w_n\| + 1 +
|\varphi(u)| + \| l \|_{Y^*} \| u_n \| +|b|.
\end{equation*}
Next, since $\{ g_n \}$, $\{ w_n \}$
and $\{ Aw_n \}$ are bounded
(recall that $A$ is a bounded operator by the pseudomonotonicity), we get
$\| A w_n - g_n \|_{Y^*}\le c_0$ and
$\| w_n \|\le c_1$ with $c_0$, $c_1 > 0$ independent of $n$.
Taking these inequalities into account,
we obtain
\begin{eqnarray*}
&&
m\, \| u_n - w_n \|^2 \le
c_0 \| u_n - w_n\| + 1 + |\varphi(u)| +
\| l \|_{Y^*} \| u_n - w_n \| +
\| l \|_{Y^*} \| w_n \| +|b| \\[2mm]
&&
\quad
\le (c_0 + \| l \|_{Y^*}) \| u_n - w_n\|
+ c_1 \| l \|_{Y^*} + 1 + |\varphi(u)| + |b|,
\end{eqnarray*}
and
\begin{equation*}
m\, \| u_n - w_n \|^2 \le
c_2 \| u_n - w_n\| + c_3,
\end{equation*}
with $c_2$, $c_3 > 0$ independent of $n$.
By the elementary inequality
\begin{equation}\label{elementary}
x^2 \le \delta_1 x + \delta_2 \ \
\Longrightarrow \ \ x^2 \le \delta_1^2 + 2 \delta_2
\end{equation}
for all $\delta_1$, $\delta_2$, $x > 0$,
we deduce that
$\{ u_n - w_n \}$ is bounded in $Y$.
Finally, $\{ u_n \}$ is bounded in $Y$ independently of $n \in \nat$.
By the reflexivity of $Y$, there are a subsequence of $\{ u_n \}$, denoted in the same way, and
$u_0 \in Y$ such that
\begin{eqnarray}\label{weakcon}
u_n \to u_0 \ \ \mbox{weakly in} \ \ Y.
\end{eqnarray}
From the condition $(m1)$ of the Mosco convergence in Definition~\ref{DefMosco}, it is clear that $u_0 \in E$.

Let $z \in E$ be arbitrary. By the condition (m2) of the Mosco convergence, there exist two sequences
$\{w_n\}$ and $\{z_n\}$ with
\begin{equation}\label{alpha1}
z_n, w_n \in E_n \ \ \mbox{such that} \ \ z_n \to z
\ \ \mbox{and} \ \ w_n \to u_0 \ \ \mbox{in} \ \ Y, \ \mbox{as} \ n\to\infty .
\end{equation}
We insert $z = w_n \in E_n$ in $(P_n)$ to get
\begin{equation*}
\langle A u_n , u_n - w_n \rangle
\le \langle  g_n, u_n - w_n \rangle
+\varphi (w_n) - \varphi (u_n).
\end{equation*}
Note that
\begin{eqnarray*}
&&\limsup_{n\to\infty} \langle A u_n , u_n - w_n \rangle=\limsup_{n\to\infty} \langle A u_n , u_n - u_0\rangle+\lim_{n\to\infty} \langle A u_n , u_0 - w_n \rangle\\[2mm]
&&=\limsup_{n\to\infty} \langle A u_n , u_n - u_0\rangle .
\end{eqnarray*}
Passing to the upper limit as $n\to \infty$ in the inequality above implies
\begin{equation*}
\limsup_{n\to\infty} \langle A u_n , u_n - u_0\rangle\le \lim_{n\to\infty}\langle  g_n, u_n - w_n \rangle
+\lim_{n\to\infty}\varphi (w_n)
- \liminf_{n\to\infty}\varphi (u_n)\le 0,
\end{equation*}
where we have used the fact that $\varphi$ is weakly semicontinuous.
Combining the last inequality with (\ref{weakcon}),
and the pseudomonotonicity of $A$,
we have
\begin{equation}\label{pscon}
\left\{
\begin{array}{lll}
Au_n \to Au_0 \ \ \mbox{weakly in} \ \ Y^*, \\ [2mm]
\langle Au_0,u_0-v\rangle\le\liminf_{n\to \infty} \langle Au_n,u_n-v\rangle\ \ \mbox{for all} \ \ v\in Y.
\end{array}
\right.
\end{equation}
Taking $z = z_n \in E_n$ in $(P_n)$, we have
\begin{equation*}
\langle A u_n , u_n - z_n \rangle
\le \langle  g_n, u_n - z_n \rangle
+\varphi (z_n) - \varphi (u_n),
\end{equation*}
which together with (\ref{alpha1}) and (\ref{pscon})
yields
\begin{eqnarray*}
&&
\langle A u_0 , u_0 - z \rangle\\[2mm]
&&
\quad
\le \liminf_{n\to \infty}\langle A u_n , u_n - z \rangle\le \limsup_{n\to\infty}\langle A u_n , u_n - z \rangle\\[2mm]
&&
\qquad
=\limsup_{n\to\infty}\langle A u_n , u_n - z \rangle+\lim_{n\to\infty}\langle A u_n , z - z_n \rangle=\limsup_{n\to\infty}\langle A u_n , u_n - z_n\rangle\\[2mm]
&&
\quad
\le \lim_{n\to\infty}\langle  g_n, u_n - z_n \rangle
+\lim_{n\to\infty}\varphi (z_n) - \liminf_{n\to\infty}\varphi (u_n)\\[2mm]
&&\qquad
\le \langle g, u_0 - z \rangle+\varphi(z)-\varphi(u_0).
\end{eqnarray*}
Because $z\in E$ is arbitrary, we deduce that
$u_0 \in E$ is a solution to Problem~$(P)$.
By the uniqueness of solution to Problem~$(P)$, we get
$$
u_0 = u.
$$
Further, the convergence of the whole sequence
follows by a contradiction argument.

Finally, we show the strong convergence of
$\{ u_n \}$ to $u$ in $Y$.
We begin with the proof of the inequality
\begin{equation}\label{beta1}
\limsup_{n\to \infty} \langle A u_n, u_n - u \rangle \le 0.
\end{equation}
From the condition $(m2)$ of the Mosco convergence
applied to $u=u_0$, we find a sequence
$\{ \eta_n \} \subset E_n$ such
that $\eta_n \to u$ in $Y$, as $n \to \infty$.
Since $u_n \in E_n$ solves Problem $(P_n)$,
we choose $z = \eta_n \in E_n$ in $(P_n)$ to
obtain
$$
\langle A u_n - g_n, \eta_n - u_n \rangle
+\varphi (\eta_n) - \varphi (u_n) \ge 0
\ \ \mbox{for all} \ \ n \in \nat,
$$
which entails
\begin{equation}\label{alpha5xx}
\langle A u_n, u_n - u \rangle
\le \langle Au_n, \eta_n - u \rangle
+ \langle g_n, u_n - \eta_n \rangle
+\varphi (\eta_n) - \varphi (u_n).
\end{equation}
We observe the following convergences
\begin{itemize}
\item[(a)]
$\langle Au_n, \eta_n - u \rangle \to 0$,
since $\eta_n - u \to 0$ in $Y$, and $A$
is a bounded operator,
\item[(b)]
$\langle g_n, u_n - \eta_n \rangle \to 0$,
by (\ref{Hyp1n})(b), and $u_n - \eta_n \to 0$ weakly in $Y$,
\item[(c)]
$\varphi(\eta_n) \to \varphi(u)$, since
$\varphi$ is a continuous function,
\item[(d)]
$\limsup_{n\to \infty}(-\varphi(u_n)) \le \varphi(u)$,
since $\varphi$ is sequentially weakly lower
semicontinuous.	
\end{itemize}
Using (a)--(d) in (\ref{alpha5xx}), we deduce (\ref{beta1}).
By the $m$-strong monotonicity of $A$, inequality (\ref{beta1}) and the fact that $u_n \to u$
weakly in $Y$, we have
\begin{eqnarray*}
&&
m \, \limsup_{n\to \infty} \| u_n - u\|^2 \le
m\, \limsup_{n\to \infty} \langle Au_n - Au, u_n - u \rangle
\\[2mm]
&&\quad
\le
m\, \limsup_{n\to \infty} \langle Au_n, u_n - u \rangle
+
  \limsup_{n\to \infty} \langle Au, u - u_n \rangle \le 0.
\end{eqnarray*}
This proves that
$\| u_n - u\| \to 0$, as $n \to \infty$.
This completes the proof of the proposition.
\end{proof}

\section{Quasi--variational--hemivariational inequality}\label{section4}

The goal of this section is to establish an existence theorem for the elliptic quasi--variational--hemivariational inequalities and demonstrate that the solution set is a compact set in $V$.
Such class of problems represent elliptic va\-ria\-tio\-nal--hemivariational
inequalities with a constraint set depending on a solution.

Let $V$ be a reflexive Banach space and $V^*$
its dual space.
The norm in $V$ and the duality brackets for the pair
$(V^*, V)$ are denoted by $\| \cdot \|$ and
$\langle \cdot, \cdot \rangle$, respectively.
Further, let $X$ be a Hilbert space with the norm $\| \cdot\|_{X}$ and the inner product $\langle \cdot, \cdot \rangle_X$.
Given a set $C \subset V$, operators
$A \colon V \to V^*$ and $M \colon V \to X$,
functions $\varphi \colon V \to \real$ and
$j \colon X \to \real$, a multifunction
$K \colon C \to 2^C$, and $f \in V^*$, we consider the following problem.

\begin{Problem}\label{mainproblem1}
{\it Find $u \in C$ such that $u \in K(u)$ and}
\begin{equation*}
\langle A u - f, z - u \rangle
+ \varphi (z) - \varphi (u)
+ j^0(Mu; Mz - Mu) \ge 0
\ \ \mbox{\rm for all} \ \ z \in K(u).
\end{equation*}
\end{Problem}

To deliver the existence of a solution to Problem~\ref{mainproblem1}, we will need the following hypotheses on the data.

\begin{itemize}
\item[${\underline{H(A)}}$]\
$A \colon V \to V^*$ is an operator such that
\begin{itemize}
\item[(a)] $A$ is pseudomonotone,
\item[(b)] $A$ is $m$-strongly monotone.
\end{itemize}
\item[${\underline{H(C)}}$] \
$C\subset V$ is nonempty, closed and convex.
\item[${\underline{H(j)}}$]\
$j \colon X \to \real$ is a locally Lipschitz function such that
	$$
	\| \partial j(x) \|_{X} \le
	\alpha + \beta \| x \|_X \ \ \mbox{for all} \ \ x \in X
	$$
with $\alpha$, $\beta \ge 0$.
\item[${\underline{H(K)}}$] \
$K \colon C \to 2^C$ is a multifunction with nonempty, closed, convex values which is weakly Mosco continuous, i.e.,
for any $\{ v_n \} \subset V$ such that $v_n \to v$ weakly in $V$, one has $K(v_n) \ \stackrel{M} {\longrightarrow} \ K(v)$, and there exists a bounded set $Q\subset V$ such that $K(u)\cap Q\neq \emptyset$ for all $u\in C$.
\item[${\underline{H(M)}}$] \
$M \colon V \to X$ is a linear  and compact operator.
\item[${\underline{H(\varphi)}}$] \
$\varphi \colon V \to \real$ is a convex, and lower semicontinuous function.
\item[${\underline{H(f)}}$] \ $f\in V^*$.
\item[${\underline{(H_0)}}$] \
$m > \beta \, \| M \|^2_{{\mathcal L}(V, X)}$.
\end{itemize}

\begin{Theorem}\label{Theorem1}
If hypotheses
$H(A)$, $H(C)$, $H(j)$, $H(K)$, $H(M)$,
$H(\varphi)$, $H(f)$, and $(H_0)$ hold, then Problem~\ref{mainproblem1} has a solution.
\end{Theorem}

\begin{proof}
The proof will be done in several steps. The main idea is to apply the  Kakutani--Ky
Fan--Glicksberg  fixed point theorem to an equivalent form of the problem.

{\bf Step 1.} First we note that Problem~\ref{mainproblem1}
can be formulated in the following equivalent form.
\begin{Problem}\label{mainproblem2}
{\it Find $u \in C$ such that $u \in K(u)$ and
there exists $w \in \partial j (Mu)$,
$w \in X$ and}
\begin{equation}\label{interineq}
\langle A u - f, z - u \rangle
+ \varphi (z) - \varphi (u)
+ \langle M^* w, z - u \rangle \ge 0
\ \ \mbox{\rm for all} \ \ z \in K(u).
\end{equation}
\end{Problem}
\noindent
Here and in what follows, $M^* \colon X^* \to V^*$ denotes the adjoint operator  to $M$.

Let $u \in C$ be a solution to Problem~\ref{mainproblem1}. Using~\cite[Proposition~3.23(iii)]{MOSbook},
we have the following
property
$$
j^0(Mu; Mz - Mu)
= \max \{ \,
\langle \zeta, Mz - Mu \rangle_{X}
\mid \zeta \in \partial j(M u)\, \}
$$
for all $z \in K(u)$. Thus, for each $z\in K(u)$, there exists
$w_z \in \partial j(Mu)$ such that
$$
j^0(Mu; Mz - Mu) = \langle w_z, Mz - Mu
\rangle_{X} = \langle M^* w_z, z-u \rangle_{X}.
$$
So, for each $z\in K(u)$ there exists $w_z \in \partial j(Mu)$ such that
\begin{equation}\label{oneineq}
\langle A u - f, z - u \rangle
+ \varphi (z) - \varphi (u)
+ \langle M^* w_z, z - u \rangle \ge 0.
\end{equation}
Arguing as in the proof of~\cite[Proposition 3.3]{khanGiannessi2000}, we are able to find an element $w \in \partial j (Mu)$, which is independent of $z\in K(u)$,
such that inequality (\ref{interineq}) holds.
Hence, we see that $u \in C$ is a solution to Problem~\ref{mainproblem2}.

Conversely, let $u \in C$ be a solution to Problem~\ref{mainproblem2}. Then there is
$w \in \partial j(Mu)$ which, by the definition of the subgradient, means that
$j^0(Mu, \eta) \ge \langle w, \eta \rangle_X$
for all $\eta \in X$. Hence,
$$
\langle M^*w, z-u \rangle_{X} =
\langle w, Mz - Mu \rangle_{X} \le
j^0(Mu; Mz - Mu)
$$
for all $z \in K(u)$. Therefore $u \in C$ solves Problem~\ref{mainproblem1}.
In conclusion, Problems~\ref{mainproblem1}
and~\ref{mainproblem2} are equivalent.
In what follows we show existence of solutions to  Problem~\ref{mainproblem2}.

{\bf Step 2.}
We obtain ``a priori" estimate for the solutions to Problem~\ref{mainproblem2}.
Let $u \in C$ be a solution to Problem~\ref{mainproblem2}.
So $u \in K(u)$ and there exists $w \in \partial j(Mu)$ such that inequality (\ref{interineq}) holds.
Since, by $H(K)$, $K(u)\cap Q$ is nonempty,
there exists $z_0 \in K(u)\cap Q$ such that
\begin{equation*}
\langle A u - f, z_0 - u \rangle
+ \varphi (z_0) - \varphi (u)
+ \langle M^* w, z_0 - u \rangle \ge 0.
\end{equation*}
Hence
\begin{equation*}
\langle A u - A z_0, u - z_0 \rangle
\le
\langle A z_0 - f + M^* w, z_0 -u \rangle
+ \varphi (z_0) - \varphi (u).
\end{equation*}
From~\cite[Proposition~5.2.25]{DMP1}, it is known that there are $\l \in Y^*$ and $b \in \real$ such that
$\varphi(z) \ge \langle l, z \rangle + b$
for all $z \in V$. Hence, we have
$- \varphi(u) \le
\| l\|_{Y^*} \| u \| + |b|$ which together with $m$-strong monotonicity of $A$ implies
$$
m\, \| u - z_0\|^2 \le \left(
\| A z_0 \|_{V^*} + \| f \|_{V^*} +
\| M \| \| w \|_{X^*}
\right)
\| u - z_0 \| + |\varphi (z_0)|
+ \| l \|_{V^*} \| u \| + |b|,
$$
where $\| M \| = \| M \|_{{\mathcal L}(V, X)}$.
We use the growth condition in $H(j)$ and the
elementary inequality
$\| u \| \le \| u - z_0 \| + \| z_0 \|$
to obtain
\begin{eqnarray*}
&&
m \, \| u - z_0\|^2 \le
\left(
\| A z_0 \|_{V^*} + \| f \|_{V^*} +
\alpha \| M \| + \| l \|_{X^*}
\right) \| u - z_0\| \\ [2mm]
&&\quad
+ \beta \| M \|^2 (\| u - z_0\| + \| z_0 \|)
\| u - z_0\|
+|\varphi (z_0)|
+ \| l \|_{V^*} \| z_0 \| + |b|,
\end{eqnarray*}
which implies
\begin{equation*}
\left( m - \beta \| M \|^2 \right)
\, \| u - z_0\|^2 \le
c_1 \, \| u - z_0 \| + c_2,
\end{equation*}
where
\begin{eqnarray*}
&&
c_1 = \| A z_0 \|_{V^*} + \| f \|_{V^*} +
\alpha \| M \| + \| l \|_{X^*}
+ \beta \| M \|^2 \| z_0 \|,
\\ [2mm]
&&\quad
c_2 = |\varphi (z_0)|
+ \| l \|_{V^*} \| z_0 \| + |b|.
\end{eqnarray*}
By the smallness condition $(H_0)$ and the elementary inequality (\ref{elementary}),
we deduce that $\| u - z_0\|$ is bounded.
So there are constants $R_1$, $R_2 > 0$
such that
$$
\| u \|\le R_1 \ \ \ \mbox{and} \ \ \
\| M u \|_X \le \| M \| \|u \| \le \| M \| R_1 =: R_2,
$$
where $R_1$ and $R_2$ rely on the set $Q$ but are independent of $z_0$, since $z_0\in Q$ and $Q$ is bounded in $V$.
This completes the proof of the desired estimates.

{\bf Step 3.}
We now define a modification of the multivalued subgradient term $\partial j$.
We introduce $F\colon X \to 2^X$ given by
$$
F(z) = \partial j (P_{R_2}(z))
\ \ \mbox{for all} \ \ z \in X,
$$
where $P_{R_2} \colon X \to X$ is the $R_2$-radial retraction defined by
$$
P_{R_2}(z) =
\displaystyle
\begin{cases}
z, &\text{if $\| z \|_X \le R_2$,} \\
\displaystyle
R_2\frac{z}{\| z\|_X}, &\text{otherwise,}
\end{cases}
\ \ \ \ \mbox{for} \ \ z \in X.
$$
Hence, we have
$$
F(z) =
\displaystyle
\begin{cases}
\partial j(z), &\text{if $\| z \|_X \le R_2$,} \\
\displaystyle
\partial j \Big(R_2\frac{z}{\| z\|_X}\Big), &\text{otherwise,}
\end{cases}
\ \ \ \ \mbox{for} \ \ z \in X.
$$
Note that $F$ has the same properties as $\partial j$ (nonempty, closed, weakly compact values, the graph of
$F$ is closed in $X \times X_w$-topology) and,
moreover, by $H(j)$, it follows
$$
\| F(z) \|_X \le \alpha + \beta R_2 =: R
\ \ \mbox{for all} \ \ z \in X.
$$

{\bf Step 4.}
Let
$D := \{ (v, w) \in C \times X \mid
\| v \| \le R_1, \ \| w\|_X \le R \}$.
For each $(v, w) \in D$ fixed, we consider
the following auxiliary problem:
\begin{equation*}
\left\{
\begin{array}{lll}
\mbox{find} \ u \in C \ \mbox{such that}
\ u \in K(v) \ \mbox{and} \\ [2mm]
\quad \langle A u - f, z - u \rangle
+\varphi (z) - \varphi (u)
+ \langle M^* w, z - u \rangle_X \ge 0 \\ [2mm]
\qquad \ \ \mbox{\rm for all} \ \ z \in K(v).
\end{array}
\right.
\leqno{P(v,w)}
\end{equation*}
Let $p \colon D \subset C \times X \to C$
be the solution map of problem $P(v, w)$ defined by
\begin{equation}\label{solutionmap}
p(v, w) = u.
\end{equation}
The map $p$ is well defined since,
for each $(v, w) \in D$,
problem $P(v, w)$ has a unique solution $u \in C$. Indeed, problem $P(v, w)$ can be equivalently written as
\begin{equation*}
\left\{
\begin{array}{lll}
\mbox{find} \ u \in C \ \mbox{such that} \ u \in E
\ \ \mbox{and} \ \
\\ [2mm]
\quad \langle A u - g, z - u \rangle
+ \varphi (z) - \varphi (u) \ge 0
\ \ \mbox{\rm for all} \ \ z \in E
\end{array}
\right.
\end{equation*}
with $g = f - M^* w$ and $E = K(v)$.
From $H(A)$, $H(K)$, $H(M)$, $H(\varphi)$, and $H(f)$,
the latter has a unique solution, see Lemma~\ref{Lemma1}.

Furthermore, we claim that
$p \colon D \subset C \times X \to C$ is continuous
from $V_w \times X_w$ to $V$.
To this end, let $(v_n, w_n) \in C \times X$,
$v_n \to v$ weakly in $V$, and
$w_n \to w$ weakly in $X$ as $n \to \infty$.
We show that
$u_n = p(v_n, w_n) \to p(v, w) = u$ in $V$.
Since, by the definition of $p$, $u_n$ is the unique solution to problem $P(v_n, w_n)$, we have
\begin{equation*}
\left\{
\begin{array}{lll}
\mbox{find} \ u_n \in C \ \mbox{such that}
\ u_n \in K(v_n) \ \mbox{and} \\ [2mm]
\quad \langle A u_n - f, z - u_n \rangle
+\varphi (z) - \varphi (u_n)
+ \langle M^* w_n, z - u_n \rangle_X \ge 0 \\ [2mm]
\qquad \ \ \mbox{\rm for all} \ \ z \in K(v_n).
\end{array}
\right.
\leqno{P(v_n, w_n)}
\end{equation*}
Again, we remark that problem $P(v_n, w_n)$ can be formulated as follows
\begin{equation*}
\left\{
\begin{array}{lll}
\mbox{find} \ u_n \in C \ \ \mbox{such that} \ \,
u_n \in E_n \ \ \mbox{and}
\\ [2mm]
\quad \langle A u_n - g_n, z - u_n \rangle
+ \varphi (z) - \varphi (u_n) \ge 0
\ \ \mbox{\rm for all} \ \ z \in E_n
\end{array}
\right.
\end{equation*}
with $g_n = f - M^* w_n \in V^*$ and
$E_n = K(v_n) \subset V$.
Using the convergences
$v_n \to v$ weakly in $V$ and
$w_n \to w$ weakly in $X$, by the compactness
of $M^*$ and $H(K)$, we get
\begin{eqnarray*}
&&
g_n \to g \ \ \mbox{in} \ \ V^*, \\[2mm]
&&
E_n = K(v_n) \ \stackrel{M} {\longrightarrow}
\ K(v) = E.
\end{eqnarray*}
By applying Proposition~\ref{Prop1}, we obtain
that $u_n \to u$ in $V$, where $u \in K(v)$
is the unique solution to problem $P(v, w)$. This proves the claim.

{\bf Step 5.}
We define the multifunction
$\Lambda \colon D \to 2^D$ by
\begin{equation}\label{alpha5}
\Lambda (v, w) =
\big(
p(v, w), F(M p(v, w))
\big)
\ \ \mbox{for} \ \ (v, w) \in D.
\end{equation}
The values of the map $\Lambda$ stay in $D$.
In fact, if
$(v, w) \in D$, then $\| v \| \le R_1$
and $\| w \|_X \le R$. Then
$\| p (v, w) \| = \| u \|\le R_1$ by Step~2,
and
$\| F(M p(v, w))\|_X  \le \alpha + \beta \| M \| \|u \| \le R$ by Step~3.
This means that
$(p(v, w), F(M p(v, w))) \in D$. Further, the values of $\Lambda$ are nonempty, closed and convex sets, by the analogous properties
of the Clarke subgradient.

The next claim is that the graph of $\Lambda$
is sequentially weakly closed in $D \times D$.
Let $(v_n, w_n) \in D$,
$({\bar{v}}_n, {\bar{w}}_n)
\in \Lambda (v_n, w_n)$,
$(v_n, w_n) \to (v, w)$ weakly in $V \times X$,
and
$({\bar{v}}_n, {\bar{w}}_n) \to ({\bar{v}}, {\bar{w}})$ weakly in $V \times X$.
We show that
$({\bar{v}}, {\bar{w}}) \in \Lambda(v, w)$.
By the definition of $\Lambda$, we have
\begin{equation}\label{SSS}
{\bar{v}}_n = p(v_n, w_n) \ \ \ \mbox{and} \ \ \
{\bar{w}}_n \in F(M p(v_n, w_n)).
\end{equation}
From the continuity of the map $p$ (see Proposition~\ref{Prop1}) and
continuity of $M$, we get
$$
p(v_n, w_n) \to p(v, w) \ \ \mbox{in} \ V
\ \ \ \mbox{and} \ \ \
M p(v_n, w_n) \to M p(v, w) \ \ \mbox{in} \ X
$$
which, together with (\ref{SSS}), entails
\begin{equation*}
{\bar{v}} = p(v, w) \ \ \ \mbox{and} \ \ \
{\bar{w}} \in F(M p(v, w)).
\end{equation*}
The latter is a consequence of the closedness of the graph of $F$ in $X \times X_w$ topology.
Hence
$
({\bar{v}}, {\bar{w}}) \in
\big(
p(v, w), F(M p(v, w))
\big) = \Lambda (v, w)$. The claim is proved.

{\bf Step 6.}
We apply the  Kakutani--Ky
Fan--Glicksberg  theorem with $Y = V \times X$ and the map $\Lambda$ given by (\ref{alpha5}).
Thus, there exists $(v^*, w^*) \in D$
such that $(v^*, w^*) \in \Lambda (v^*, w^*)$.
Hence, $v^* = u^*$ and
$w^* \in F(M u^*)$, where
$u^* \in C$ satisfies $u^* \in K(u^*)$ and
$$
\langle A u^* - f, z - u^* \rangle
+ \varphi (z) - \varphi (u^*)
+ \langle M^* w^*, z - u^* \rangle_X
\ge 0 \ \ \mbox{\rm for all} \ \ z \in K(u^*)
$$
with $w^* \in F(M u^*)$.
Repeating the estimate of Step~2 for
$u^* \in C$, we have
$\| u^*\| \le R_1$ and $\| M u^* \|_X \le R_2$
which implies
$F(M u^*) = \partial j (M u^*)$. Hence,
$u^* \in C$ is such that $u^* \in K(u^*)$ and
$$
\langle A u^* - f, z - u^* \rangle
+ \varphi (z) - \varphi (u^*)
+ \langle M^* w^*, z - u^* \rangle_X
\ge 0 \ \ \mbox{\rm for all} \ \ z \in K(u^*)
$$
with $w^* \in \partial j(M u^*)$.
Thus $u^* \in C$ solves Problem~\ref{mainproblem2}. By Step~1, we conclude that
$u^* \in C$ is a solution to Problem~\ref{mainproblem1}.
The proof is complete.
\end{proof}

The following result concerns the compactness of the solution set.
\begin{Theorem}\label{Theorem2}
If the hypotheses of Theorem~\ref{Theorem1} hold, then the solution set to Problem~\ref{mainproblem1} is a nonempty and compact subset of $V$.
\end{Theorem}
\begin{proof}
Let ${\mathbb S}$ be the solution set to Problem~\ref{mainproblem1}. It is a nonempty set by Theorem~\ref{Theorem1}.
From the proof of Theorem~\ref{Theorem1}, it follows that  ${\mathbb S} \subset p(D)$, where the solution map
$p \colon D \subset C \times X \to C$
is defined by (\ref{solutionmap}) with
$$
D := \{ (v, w) \in C \times X \mid
\| v \| \le R_1, \ \| w\|_X \le R \}.
$$
We establish compactness of ${\mathbb S}$ based on two claims below.

{\bf Claim 1.} \
The set $p(D)$ is compact in $C$.
It suffices to show that $p(D)$ is sequentially compact, that is, any sequence extracted from $p(D)$
contains a subsequence that converges to an element
in this set.
Let $\{ u_n \} \subset p(D)$. From the definition of $p$, we find a sequence $\{ (v_n, w_n) \} \subset D$ such that
$u_n = p(v_n, w_n)$ for all $n \in \nat$.

Note that $D$ is bounded, closed and convex in the reflexive Banach space $V \times X$, hence it is sequentially weakly compact. So we may assume, by passing to a subsequence if necessary, that
$(v_n, w_n) \to (v, w)$ weakly in $V \times X$
with $(v, w) \in D$.
Since $p$ is sequentially continuous from $V_w \times X_w$ to $C$, we have
$$
p(v_n, w_n) \to p(v, w) \ \ \mbox{in} \ \ V.
$$
This means that $u_n := p(v_n, w_n)$
and $u_0 := p(v, w)$ satisfy $u_n \to u_0$ in $V$,
as $n \to \infty$, and $u_0 \in p(D)$. Hence, the claim follows.

{\bf Claim 2.} \
The set ${\mathbb S}$ is sequentially closed in $V$.
Let $\{ u_n  \} \subset {\mathbb S}$ be such that
$u_n \to {\bar{u}}$ in $V$, as $n \to \infty$ with
${\bar{u}} \in V$. We will show that
${\bar{u}} \in {\mathbb S}$.
It is clear that ${\bar{u}} \in C$.
We have
$u_n \in C$, $u_n \in K(u_n)$ and
\begin{eqnarray}\label{111}
&&
\langle A u_n - f, z - u_n \rangle
+ \varphi (z) - \varphi (u_n)
\nonumber \\ [2mm]
&&
\quad
+ j^0(Mu_n; Mz - Mu_n) \ge 0
\ \ \mbox{\rm for all} \ \ z \in K(u_n).
\end{eqnarray}
Let $w \in K({\bar{u}})$ be arbitrary.
Since $u_n \in C$, by the condition $(m2)$
of the Mosco convergence
$K(u_n) \, \stackrel{M} {\longrightarrow} \, K({\bar{u}})$, there exists $w_n \in C$ such that
$w_n \in K(u_n)$ and $w_n \to w$ in $V$,
as $n \to \infty$.
We choose $z = w_n \in K(u_n)$ in (\ref{111})
to get
\begin{equation}\label{222}
\langle A u_n - f, w_n - u_n \rangle
+ \varphi (w_n) - \varphi (u_n)
+ j^0(Mu_n; Mw_n - Mu_n) \ge 0.
\end{equation}
Next, we will pass to the limit in (\ref{222}).
The following convergences hold:
\begin{itemize}
	\item[(i)]
	$Au_n \to A{\bar{u}}$ weakly in $V^*$, it is a consequence of~\cite[Proposition~3.66]{MOSbook}
	and the facts that $A$ is a bounded operator and
	$\langle Au_n, u_n - {\bar{u}} \rangle = 0$,
	\item[(ii)]
	$M u_n \to M {\bar{u}}$ in $X$, by hypothesis $H(M)$,
	\item[(iii)]
	$\varphi(w_n) - \varphi(u_n) \to \varphi(w) - \varphi({\bar{u}})$, since
	$\varphi$ is a continuous function,
	\item[(iv)]
	$\limsup_{n\to \infty}j^0(Mu_n; Mw_n - Mu_n) \le
	j^0(M {\bar{u}}; Mw - M{\bar{u}})$, since $j^0(\cdot; \cdot)$ is upper semicontinuous, see~\cite[Proposition~3.23(ii)]{MOSbook}.
\end{itemize}
We employ the convergences (i)--(iv) and take the upper limit in the  inequality (\ref{222}) to obtain
\begin{equation*}
\langle A {\bar{u}} - f, w - {\bar{u}} \rangle
+ \varphi (w) - \varphi ({\bar{u}})
+ j^0(M {\bar{u}}; Mw - M{\bar{u}}) \ge 0
\end{equation*}
for all $w \in K({\bar{u}})$.
Finally, since $u_n \in K(u_n)$,
$u_n \to {\bar{u}}$ in $V$, by the condition $(m1)$,
we deduce ${\bar{u}} \in K({\bar{u}})$.
Hence ${\bar{u}} \in {\mathbb S}$ and the claim is proved.

Since a closed subset of a compact set is compact,
by Claims~1 and~2, we conclude that
the solution set ${\mathbb S}$ of  Problem~\ref{mainproblem1} is nonempty and compact subset of $V$.
The proof is complete.
\end{proof}

We complete this section by providing a simple example of function $j$ which satisfies $H(j)$ and whose subgradient is not relaxed monotone.
This means that Theorems~\ref{Theorem1} and~\ref{Theorem2}
are applicable in this case while the results of~\cite{MKZ2020}
can not be used.
\begin{Remark}\label{RRCC}
We recall that a locally Lipschitz function $j \colon X \to \real$ satisfies the relaxed monotonicity condition, if
\begin{equation}\label{EXX1}
\exists \, m \ge 0 \ \ \mbox{such that} \ \
\langle \partial j(x_1) - \partial j(x_2), x_1 - x_2 \rangle_X
\ge -m \, \| x_1 - x_2\|^2_X
\end{equation}
for all $x_1$, $x_2 \in X$, see~\cite[Section~3.3]{MOSbook}.
For a convex function $j$, condition $(\ref{EXX1})$
re\-du\-ces to the monotonicity of the (convex) subdifferential,
i.e., $m = 0$.
Now, let $j \colon \real \to \real$ be defined by
\begin{equation*}
j(r) =
\begin{cases}
0 &\text{\rm if \ $r < 0$,} \\
\frac{r^2}{2} &\text{\rm if \ $r \in [0, 1),$} \\
1 &\text{\rm if \ $r \ge 1$,}
\end{cases}
	\ \ \mbox{for} \ \ r\in \real.
\end{equation*}
It is clear that $j$ is a locally Lipschitz,
nonconvex function with
\begin{equation*}
	\partial j(r) =
	\begin{cases}
	0 &\text{\rm if \ $r < 0$,} \\
	r &\text{\rm if \ $r \in [0, 1),$} \\
	[0, 1] &\text{\rm if \ $r  = 1,$} \\
	0 &\text{\rm if \ $r > 1$}
	\end{cases}
	\ \ \mbox{for} \ \ r\in \real,
\end{equation*}
and $| \partial j(r)| \le 1$ for all $r \in \real$, so $H(j)$ holds with $\alpha =1$ and $\beta = 0$.
Note that $j$ satisfies (\ref{EXX1}) if and only if
$\partial \phi$ is a monotone graph, i.e.,
\begin{equation}\label{EXX2}
\exists \, m \ge 0 \ \ \mbox{such that} \ \
(\partial \phi(r) - \partial \phi(s)) (r-s) \ge 0
\ \ \mbox{for all} \ \ r, s \in \real,
\end{equation}
where $\phi(r) = j(r) + \frac{m}{2} r^2$ for $r \in \real$.
We claim that $j$ does not satisfy (\ref{EXX2}).
Suppose, contrary to our claim, that (\ref{EXX2}) holds.
Let
$$
r_n = 1 - \frac{1}{n}, \ \
s_n = 1 + \frac{1}{n} \ \ \mbox{for} \ \ n \in \nat.
$$
Then, there exists $m \ge 0$ such that for all $n \in \nat$,
we have
$(\partial \phi(r_n) - \partial \phi(s_n)) (r_n-s_n) \ge 0$.
Hence, a short computation leads to
$n \le m+1$, a contradiction.
Note also that in this example, the smallness condition $(H_0)$
is trivially satisfied.
\end{Remark}

\section{Applications to semipermeability models}\label{section5}

In this section we study two applications of elliptic quasi--variational--hemivariatio\-nal inequalities for the semipermeable media.
The general inequalities of this kind incorporate both the
interior and boundary semipermeability.
In both applications, we simultaneously treat monotone
and nonmonotone relations described by subdifferential operators.
In the first model the interior semipermeability law is governed by the subgradient of a locally Lipschitz potential while the heat flux on a part of the boundary is formulated as a monotone relation of the temperature.
In the second model the boundary semipermeability condition is nonmonotone and involves unilateral constraints while the interior semipermeability is described by a subdifferential of a convex function.
The weak formulations of both models turn out to be  quasi--variational--hemivariational inequalities of elliptic type.


Let $\Omega$ be a bounded domain of $\real^d$
(that is, $\Omega$ is an open, bounded and connected set)
with Lipschitz continuous boundary $\partial \Omega = \Gamma$.
Let $\nu$ denote the unit outward normal vector on $\Gamma$ which exists a.e. on $\Gamma$.
We suppose the boundary consists of measurable and disjoint parts $\Gamma_{1}$ and $\Gamma_{2}$ such that $m(\Gamma_{1}) > 0$ and
$\Gamma = {\bar{\Gamma}}_1 \cup {\bar{\Gamma}}_2$.

The classical model for the heat conduction problem is described by the following boundary value problem.

\begin{Problem}\label{problemsemi}
	Find a temperature $u \colon \Omega \to \mathbb{R}$ such that $u \in U(u)$ and
	\begin{eqnarray}
	-{\rm div}\, a(x, \nabla u(x)) = g(x,u(x))
	&&\ \ \mbox{\rm in} \ \ \Omega
	\label{semi1}\\ [2mm]
	\left\{\begin{array}{lll}
    g(x,u(x)) = g_1(x) + g_2(x,u(x)), \\[2mm]
    -g_2(x,u(x)) \in \partial h(x, u(x))
    \end{array}\right.
	&&\ \ \mbox{\rm in} \ \ \Omega
	\label{semi2}\\ [2mm]
		u(x) = 0  &&\ \ \mbox{\rm on} \ \ \Gamma_{1} \label{semi4} \\[2mm]
	\displaystyle-\frac{\partial u(x)}{\partial \nu_a}
	\in \partial_c k (x, u(x)) &&\ \ \mbox{\rm on} \ \ \Gamma_{2}, \label{semi5}
	\end{eqnarray}
where $U\colon V\to 2^V$ is defined by
\begin{equation}\label{defU}
U(u):=\{\, w \in V \mid r(w)\le m(u) \, \}.
\end{equation}
\end{Problem}
Here, the conormal derivative associated with the field $a$
is given by
$$
\frac{\partial u(x)}{\partial \nu_a} =
a(x, \nabla u(x)) \cdot \nu .
$$

We will provide the variational formulation of
Problem~\ref{problemsemi} within the framework of Section~\ref{section4}.
We introduce the following space
\begin{equation}\label{spaceV}
V = \{ \, v \in H^1(\Omega)
\mid v = 0 \ \mbox{on} \ \Gamma_{1}\, \}.
\end{equation}
\noindent
Since $m(\Gamma_{1}) > 0$, on $V$ we can consider the norm
$\| v \|_V = \| \nabla v\|_{L^2(\Omega; \real^d)}$
for $v \in V$ which is equivalent on $V$ to the standard $H^1(\Omega)$ norm due to Korn's inequality. By $\gamma \colon V \to L^2(\Gamma)$ we denote the trace operator which is known to be linear, bounded and compact. Moreover, by $\gamma v$ we denote the trace of an element $v \in H^1(\Omega)$. In the sequel, we denote by $i\colon V\to L^2(\Omega)$
the embedding operator from $V$ to $L^2(\Omega)$
and $\| i \| = \| i \|_{{\cal L}(V,L^2(\Omega))}$ stands for its norm.
We also set $C = V$ and $X = L^2(\Omega)$.

In order to study the variational formulation of Problem~\ref{problemsemi}, we need the following hypotheses.
\begin{eqnarray}
\left\{
\begin{array}{ll}
a \colon \Omega \times \real^d \to \real^d
\ \mbox{is such that}
\\ [2mm]
\ \ {\rm (a)} \ a(\cdot, \xi) \ \mbox{is measurable on}
\ \Omega \ \mbox{for all} \ \xi \in \real^d, \\ [1mm]
\qquad \mbox{and} \ a(x, 0) = 0 \ \mbox{for a.e.}
\ x \in \Omega.
\\ [2mm]
\ \ {\rm (b)} \
a(x, \cdot) \ \mbox{is continuous on} \ \real^d \ \mbox{for a.e.} \ x \in \Omega.
\\ [2mm]
\ \ {\rm (c)} \ \| a(x, \xi) \| \le
m_a \, (1 + \| \xi \|)
\ \ \mbox{for all} \ \ \xi \in \real^d, \ \mbox{a.e.} \ x \in \Omega \\ [1mm]
\qquad \mbox{with} \ m_a > 0. \\ [2mm]
\ \ {\rm (d)} \
(a(x, \xi_1) - a(x, \xi_2)) \cdot (\xi_1 - \xi_2)
\ge \alpha_a \, \| \xi_1 - \xi_2 \|^2 \\ [1mm]
\qquad \mbox{for all} \ \ \xi_1, \xi_2 \in \real^d,
\ \mbox{a.e.} \ x \in \Omega \ \ \mbox{with} \ \alpha_a > 0.
\end{array}
\right.
\label{functiona}
\end{eqnarray}
\begin{eqnarray}
\left
\{\begin{array}{ll}
h \colon \Omega \times \real \to \real \ \mbox{is such that}
\\ [2mm]
\ \ {\rm (a)} \
h(\cdot,r) \ \mbox{is measurable on} \ \Omega
\ \mbox{for all} \ r \in \mathbb{R}
\ \mbox{and there}\ \ ~ \  \\ [2mm]
\ \ \qquad \mbox{exists} \
\overline{e}\in L^{2}(\Omega) \ \mbox{such that} \
h(\cdot,\overline{e}(\cdot))\in L^{1}(\Omega). \\ [2mm]
\ \ {\rm (b)} \
h(x,\cdot) \ \mbox{is locally Lipschitz on} \ \mathbb{R},
\ \mbox{ for a.e.} \ x \in \Omega. \\ [2mm]
\ \ {\rm (c)} \
\mbox{there exist} \ \overline{c}_{0},\, \overline{c}_{1}\geq 0
\ \mbox{such that} \\ [2mm]
\ \ \qquad |\partial h(x,r)|\le
\overline{c}_{0} + \overline{c}_{1}| r | \ \mbox{for all}
\  r \in \mathbb{R}, \ \mbox{a.e.} \
x \in \Omega.
\end{array}
\right.
\label{functionj}
\end{eqnarray}	
\begin{eqnarray}
\left\{
\begin{array}{ll}
k \colon \Gamma_2 \times \real \to \real
\ \mbox{is such that}
\\ [2mm]
\ \ {\rm (a)} \
k(\cdot,r) \ \mbox{is measurable on} \ \Gamma_2
\ \mbox{for all} \ r \in \mathbb{R}. \qquad\qquad \ \ \\ [2mm]
\ \ {\rm (b)} \
k(x,\cdot) \ \mbox{is convex and continuous on} \ \real,
\ \mbox{a.e.} \ x \in \Omega. \\ [2mm]
\ \ {\rm (c)} \
x\mapsto k(x,v(x)) \ \, \mbox{belongs to} \ \, L^1(\Gamma_2)
\ \, \mbox{for} \ \, v\in L^2(\Gamma_2). \ ~\
\end{array}
\right.
\label{functiong}
\end{eqnarray}
\begin{equation}\label{HYPOTU}
\left\{\begin{array}{ll}
r\colon V\to \real\mbox{ is positively homogeneous, subadditive and} \ ~~ \ \quad ~~ \ \quad
\\[2mm]
\ \ \quad \mbox{lower semicontinuous},\\[2mm]
m\colon L^2(\Omega)\to \real\mbox{ is continuous such that}
\\[2mm]
\ \ \quad \rho:=\inf_{v\in L^2(\Omega)}m(v)>0
\ \mbox{and} \ r(0)\le \rho,
\\[2mm]
g_1 \in L^2(\Omega).
\end{array}\right.
\end{equation}

Using the standard procedure, we obtain the weak formulation of Problem~\ref{problemsemi} which is a quasi--variational--hemivariational inequality.
\begin{Problem}\label{problemsemi2}
Find $u \in V$ such that $u \in U(u)$ and
\begin{eqnarray*}
&&
\int_{\Omega} a(x, \nabla u(x)) \cdot \nabla (z(x)-u(x)) \, dx +
\int_{\Gamma_{2}} \big(
k(x, z(x)) - k (x, u(x)) \big) \, d\Gamma \\ [1mm]
&&
\quad +\int_{\Omega} h^0(x, u(x); z(x) - u(x)) \, dx \ge
\int_{\Omega} g_1(x) (z(x) - u(x))\, dx
\ \ \mbox{\rm for all} \ \ z \in U(u).
\end{eqnarray*}
\end{Problem}

\begin{Theorem}\label{existencesemi4}
	Assume that conditions (\ref{functiona})--(\ref{HYPOTU}) are satisfied and the following smallness condition holds
	\begin{equation}\label{semi_smallness}
	\overline c_1\sqrt{2}\, \| i \|^2 < \alpha_{a}.
	\end{equation}
	Then, Problem~\ref{problemsemi2} has a nonempty and compact solution set in $V$.
\end{Theorem}
\begin{proof}
We apply Theorems~\ref{Theorem1} and~\ref{Theorem2}
with the following data: $C = V$,
$X = L^2(\Omega)$,
$K(\cdot) = U(\cdot)$, $f = g_1$ and
\begin{eqnarray}
A \colon V\to V^*, \quad
&&\langle A u,v \rangle =
\int_{\Omega} a(x, \nabla u(x)) \cdot \nabla v(x)\, dx
\ \ \mbox{for} \ \ u, v \in V, \label{semiA1}
\\ [2mm]
\varphi \colon V\to \mathbb{R}, \quad
&&\varphi(v)=
\int_{\Gamma_{2}} k(x,v(x)) \, d\Gamma
\ \ \mbox{for}\ \ v \in V,\label{semiB1}
\\ [1mm]
j \colon X \to \mathbb{R}, \quad
&&j(w)=
\int_{\Omega} h(x,w(x))\, dx \ \ \mbox{for} \ \
w \in X,\label{semij1} \\[1mm]
M \colon V \to X, \quad
&&
M = i \colon V \to X \ \ \mbox{the embedding map}. \label{semiM1}
\end{eqnarray}

First, we note that by~\cite[Theorem 3.47]{MOSbook} and hypothesis (\ref{functionj}), $j\colon X \to \real$ is a locally Lipschitz function such that
\begin{eqnarray}
&&
\label{*}
\displaystyle j^0(w;z)\le \int_\Omega h^0(x,w(x);z(x))\,dx, \\[2mm]
&&
\label{**}
\displaystyle\partial j(w)\subset \int_\Omega \partial h(x,w(x))\,dx
\end{eqnarray}
for all $w$, $z\in X$.
Under our notation, from (\ref{*}) it is clear that any solution of Prob\-lem~\ref{mainproblem1} is also a solution of  Problem~\ref{problemsemi2}. Based on this fact, we are going to check the validity of the hypotheses
$H(A)$, $H(C)$, $H(j)$, $H(K)$, $H(M)$,
$H(\varphi)$, $H(f)$, and $(H_0)$
of Theorem~\ref{Theorem1} and show the existence of a solution to  Problem~\ref{problemsemi2}.

Second, by virtue of the definition of $A$,
see (\ref{semiA1}), and conditions (\ref{functiona})(a)--(b), we can see that $A$ is continuous. Besides, hypothesis (\ref{functiona})(d) reveals that $A$ is strongly monotone with constant $m=m_a$. The boundedness of $A$ can be verified directly by applying H\"older inequality and (\ref{functiona})(c). This means that $A$ is pseudomonotone (see~\cite[Proposition 3.69]{MOSbook}).
Thus $H(A)$ holds.

Employing the growth condition (\ref{functionj})(c),
inclusion (\ref{**}), and the H\"older inequality,
we deduce that
$j$ satisfies condition $H(j)$ with
$\beta =\overline c_1\sqrt{2}$.
From the definition of $\varphi$, we can see that it is convex and continuous which implies $H(\varphi)$.
Furthermore, condition $H(M)$ holds by the well-known properties of the embedding operator.
It is also obvious that $H(C)$ and $H(f)$ are satisfied.

Next, we shall prove that the multifunction
$K$ fulfills hypothesis $H(K)$.
Set $Q=\{0_V\}$ which is a zero element of $V$.
Because of the condition
$r(0)\le \rho:=\inf_{v\in L^2(\Omega)}m(v)$,
for each $v\in V$, we have
\begin{equation*}
0_V\in K(v)\cap Q\ \ \mbox{for all} \ \ v\in V.
\end{equation*}
For any $v\in V$ fixed, let $u$, $w\in K(v)$ and $\lambda \in(0,1)$ be arbitrary. The convexity of $r$ (since $r$ is positively homogeneous and subadditive) entails
\begin{equation*}
r(\lambda u+(1-\lambda)w)\le \lambda r(u)+(1-\lambda)r(w)
\le \lambda m(v)+(1-\lambda) m(v) = m(v),
\end{equation*}
this means $\lambda u+(1-\lambda)w\in K(v)$, i.e.,
$K(v)$ is convex.
Let $\{u_n\}\subset K(v)$ be such that $u_n\to u$ as $n\to \infty$ for some $u\in V$. The continuity of $r$ implies
\begin{equation*}
r(u)\le \liminf_{n\to \infty} r(u_n)\le m(v).
\end{equation*}
Hence, $K(v)$ is a closed set for each $v\in V$.
Thus, the multifunction $K \colon V \to 2^V$ has nonempty, closed, and convex values.

Next, we fix a sequence $\{v_n\}\subset V$ which is such that
\begin{equation}
v_n\to v\ \mbox{ weakly in $V$ as $n\to \infty$}
\end{equation}
for some $v\in V$.
We shall verify that $K(v_n) \, \stackrel{M} {\longrightarrow} \, K(v)$ by checking the conditions (m1) and (m2) of Definition~\ref{DefMosco}.
For the proof of (m1),
let $u\in K(v)$ be arbitrary and we set  $u_n=\frac{m(v_n)}{m(v)}u$.
Then,
from the positive homogeneity of $r$ and the condition $\rho>0$, it follows
\begin{equation*}
r(u_n)=\frac{m(v_n)}{m(v)} r(u)\le m(v_n),
\end{equation*}
which gives
$u_n\in K(v_n)$ for every $n\in\mathbb N$.
Since $V$ embeds into $X$ compactly, and $m$ is continuous,
a short calculation gives
\begin{equation*}
\lim_{n\to \infty}\|u_n-u\|
= \lim_{n\to \infty}\left\|\frac{m(v_n)}{m(v)}u-u\right\|
=\lim_{n\to \infty}\frac{|m(v_n)-m(v)|}{m(v)}\|u\|=0,
\end{equation*}
and so, $u_n\to u$ in $V$ as $n\to \infty$. This proves (m1).
Next, we show the condition (m2).
Let
$\{u_n \}\subset V$ be such that
\begin{equation*}
u_n\in K(v_n),\mbox{ and $u_n \to u$ weakly in $V$
as $n\to \infty$}
\end{equation*}
for some $u \in V$.
The inequalities
\begin{equation*}
r(u)\le \liminf_{n\to \infty}r(u_n)
\le \liminf_{n\to \infty}m(v_n)=m(v),
\end{equation*}
imply $u\in K(v)$, where we have used the weak lower semicontinuity of $r$, the continuity of $m$ and the compactness of the embedding of $V$ into $X$. Hence (m2) follows. This means that condition $H(K)$ holds.
Additionally, $(H_0)$ is a consequence of (\ref{semi_smallness}).

Therefore, all hypotheses of Theorem~\ref{Theorem1} are verified. We are now in a position to invoke this theorem to conclude that the set of solutions to Problem~\ref{mainproblem1} is nonempty. This means that Problem~\ref{problemsemi2} has at least one solution. Moreover, arguing as in the proof of Theorem~\ref{Theorem2}, we can demonstrate that the solution set of Problem~\ref{problemsemi2} is compact in $V$.
\end{proof}

\begin{Example}
Let $r\colon V\to\real$ and $m\colon L^2(\Omega)\to \real$
be defined by
\begin{equation*}
r(v):=\int_\Omega|\nabla v(x)|\varrho_1(x)\,dx
\ \ \mbox{\rm for all} \ \ v \in V,
\end{equation*}
and
\begin{equation*}
m(v):=m_0+\int_\Omega|v(x)|\varrho_2(x)\,dx
\ \ \mbox{\rm for all} \ \ v \in L^2(\Omega),
\end{equation*}
where $\varrho_1$, $\varrho_2 \in L^2(\Omega)$,
$\varrho_1$, $\varrho_2 \ge 0$ and $m_0 > 0$.
It is not difficult to prove that the functions
$r$ and $m$ satisfy the condition (\ref{HYPOTU}).
\end{Example}

We complete this section by presenting another mixed boundary value problem with a unilateral condition on a part of the boundary
to which our abstract results of Section~\ref{section3} could be applied.
This model describes a boundary and interior semipermeability problem.
	
Let $\Omega$ be a bounded domain of $\real^d$, $d=2$, $3$ with Lipschitz continuous boundary $\partial \Omega = \Sigma$ which consists of disjoint measurable parts $\Sigma_{1}$, $\Sigma_2$ and $\Sigma_{3}$ such that
$\Sigma = {\bar{\Sigma}}_1 \cup {\bar{\Sigma}}_2 \cup {\bar{\Sigma}}_3$ and $m(\Sigma_{1}) > 0$.
\begin{Problem}\label{problemsemi3}
	Find a function $u \colon \Omega \to \mathbb{R}$ such that $u \in U(u)$ and
	\begin{eqnarray*}
	-{\rm div}\, a(x, \nabla u(x))+\partial_cp(x,u(x)) \ni g_1(x)
	&&\ \ \mbox{\rm in} \ \ \Omega,
	\\ [2mm]
	u = 0  &&\ \ \mbox{\rm on} \ \ \Sigma_{1},  \\[2mm]
	-\frac{\partial u(x)}{\partial \nu_a}
	\in \partial h_2 (x, u(x)) &&\ \ \mbox{\rm on} \ \ \Sigma_{2},\\[2mm]
	\left\{\begin{array}{lll}
    u(x)\le k_2(x),\\[2mm]
	\displaystyle\frac{\partial u(x)}{\partial \nu_a}\le 0,\\[2mm]
    \displaystyle\frac{\partial u(x)}{\partial \nu_a}(u(x)-k_2(x))=0,
    \end{array}\right.
&&\ \ \mbox{\rm on} \ \ \Sigma_{3}, \\
	\end{eqnarray*}
where $U\colon C\to 2^C$ is defined by
\begin{equation*}\label{defUs}
U(u):=\{ \, w \in C \mid r(w)\le m(u) \, \},
\end{equation*}
and the set $C\subset V$ is given by
\begin{equation*}
C:=\{ \, v\in V \mid v(x)\le k_2(x)
\ \ \mbox{\rm on}\ \ \Sigma_3 \, \}.
\end{equation*}
\end{Problem}

We impose the following assumptions on the data.
\begin{eqnarray}
\left
\{\begin{array}{ll}
h_2 \colon \Sigma \times \real \to \real \ \mbox{is such that}
\\ [2mm]
\ \ {\rm (a)} \
h_2(\cdot,r) \ \mbox{is measurable on} \ \Sigma
\ \mbox{for all} \ r \in \mathbb{R}
\ \mbox{and there} \ \ \ \\ [2mm]
\ \ \qquad \mbox{exists} \
\overline{e}\in L^{2}(\Sigma) \ \mbox{such that} \
h_2(\cdot,\overline{e}(\cdot))\in L^{1}(\Sigma). \\ [2mm]
\ \ {\rm (b)} \
h_2(x,\cdot) \ \mbox{is locally Lipschitz on} \ \mathbb{R},
\ \mbox{ for a.e.} \ x \in \Sigma. \\ [2mm]
\ \ {\rm (c)} \
\mbox{there exist} \ \overline{c}_{0},\, \overline{c}_{1}\geq 0
\ \mbox{such that} \\ [2mm]
\ \ \qquad |\partial h_2(x,r)|\le
\overline{c}_{0} + \overline{c}_{1}| r | \ \mbox{for all}
\  r \in \mathbb{R}, \ \mbox{a.e.} \
x \in \Sigma.
\end{array}
\right.
\label{functionj2}
\end{eqnarray}
\begin{eqnarray}
\left\{
\begin{array}{ll}
p \colon \Omega \times \real \to \real
\ \mbox{is such that}
\\ [2mm]
\ \ {\rm (a)} \
p(\cdot,r) \ \mbox{is measurable on} \ \Omega
\ \mbox{for all} \ r \in \mathbb{R}. \qquad\qquad \ \ \\ [2mm]
\ \ {\rm (b)} \
p(x,\cdot) \ \mbox{is convex and continuous on} \ \real,
\ \mbox{a.e.} \ x \in \Omega. \\ [2mm]
\ \ {\rm (c)}\
\mbox{$x\mapsto p(x,v(x))$ belongs to $L^1(\Omega)$ for each $v\in L^2(\Omega)$.}
\end{array}
\right.
\label{functionps}
\end{eqnarray}
\begin{equation}\label{funck2}
 k_2\in L^2(\Sigma_3), \ k_2 \ge 0 \mbox{ and }k_2\not\equiv0.
\end{equation}

Using the Green formula, we arrive at the variational formulation of Problem~\ref{problemsemi3}.
\begin{Problem}\label{problemsemi4}
	Find $u \in C$ such that $u \in U(u)$ and
	\begin{eqnarray*}
		&&
		\int_{\Omega} a(x, \nabla u(x)) \cdot \nabla (z(x)-u(x)) \, dx +
		\int_{\Sigma_{2}}
	h_2^0(x, u(x);z(x)-u(x)) \, d\Gamma \\ [1mm]
		&& \qquad +\int_{\Omega} p(x, z(x))-p(x,u(x)) \, dx \ge
		\int_{\Omega} g_1(x) (z(x) - u(x))\, dx
	\end{eqnarray*}
for all $z \in U(u)$.
\end{Problem}

For Problem~\ref{problemsemi4} we need the following notation.
Let the space $V$ be defined by (\ref{spaceV}),
$X = L^2(\Sigma)$,
$M = \gamma \colon V \to X$ be the trace operator,
$A \colon V \to V^*$ be an operator given by (\ref{semiA1}),
$f = g_1$, and
$K(\cdot) = U(\cdot)$. Let the functions
$\varphi \colon V \to \real$ and
$j \colon X \to \real$ be defined by
$$
\varphi(v)=
\int_{\Omega} p(x,v(x)) \, dx
\ \ \mbox{for}\ \ v \in V,
$$
$$
j(w)=
\int_{\Sigma_2} h_2(x,w(x))\, d\Gamma \ \ \mbox{for} \ \
w \in X,
$$
respectively.
Under the above notation Problem~\ref{problemsemi4}
enters the abstract framework of Section~\ref{section3}.
Arguing as in the proof of Theorem~\ref{existencesemi4}, we obtain the following result.
\begin{Theorem}\label{existencesemi5}
	Assume that conditions   (\ref{functiona}), (\ref{HYPOTU}), and (\ref{functionj2})--(\ref{funck2}) are satisfied and the following smallness condition holds
	\begin{equation*}\label{semi_smallness2}
	\overline c_1\sqrt{2}\, \| \gamma\|^2 < \alpha_{a}.
	\end{equation*}
	Then, Problem~\ref{problemsemi4} has a nonempty and compact solution set in $V$.
\end{Theorem}

\section*{Acknowledgements}
Project is supported by the European Union's Horizon 2020 Research and Innovation 
Programme under the Marie Sk{\l}odowska-Curie grant agreement No. 823731 CONMECH, 
National Science Center of Poland under Preludium Project No. 2017/25/N/ST1/00611, 
NNSF of China Grant Nos. 12001478 and 12026256, and the Startup Project of Doctor 
Scientific Research of Yulin Normal University No. G2020ZK07. It is also supported by 
Natural Science Foundation of Guangxi Grant Nos. 2021GXNSFFA196004, 2020GXNSFBA297137 
and 2018GXNSFAA281353, and by the Beibu Gulf University under project No. 2018KYQD06. 
The first author is also supported by the projects financed by the Ministry of Science and Higher 
Education of Republic of Poland under Grants Nos. 4004/GGPJII/H2020/2018/0 and
440328/PnH2/2019, and by the National Science Centre of Poland under Project
No. 2021/41/B/ST1/01636.		

%
%
%
%
%
%
%
%
%
%
%
%
%
%
%
%

\end{document}